\documentclass[10pt,b5paper]{article}
\usepackage[T1]{fontenc}
\usepackage[utf8]{inputenc}
\usepackage{amsmath,amssymb,amsthm}
\usepackage{geometry}
\newcommand{\E}{\mathbb E}
\newcommand{\Z}{\mathbb Z}
\newcommand{\R}{\mathbb R}
\newcommand{\Mult}{\operatorname{Mult}}

\theoremstyle{plain}
\newtheorem{theorem}{Theorem}[section]
\newtheorem{proposition}[theorem]{Proposition}
\newtheorem{corollary}[theorem]{Corollary}

\theoremstyle{remark}
\newtheorem{remark}[theorem]{Remark}

\title{Multinomial probabilities near the mode:\\
integer modes and the complete local expansion}
\author{
N.~Elezovi\'c\\
Department of Applied Mathematics,\\
Faculty of Electrical Engineering and Computing,\\
University of Zagreb, 10000 Zagreb, Croatia\\
\texttt{neven.elezovic@fer.hr}
}
\date{July 2026}

\begin{document}
\maketitle

\begin{abstract}
	Let $X\sim\Mult(N;p_1,\dots,p_n)$ with $p_i>0$ and $\sum_ip_i=1$.  We study
	$\Pr\{X=m\}$ for lattice points near $Np$ from two complementary viewpoints.  First, we give an
	exact integer-mode criterion: the multinomial mode is the Jefferson--D'Hondt divisor
	apportionment, given when the mode is unique by $m_i=\lfloor (N+d)p_i\rfloor$, $d\in[0,n)$.  For
	$n\ge3$ the mode need not be obtained by rounding each $Np_i$ to a neighbouring integer; the
	excess $m_i-\lceil Np_i\rceil$ can reach $n-2$, so the two-sided localisation
	$|m_i-Np_i|<1$ stated in a standard reference is not valid in general.  Second, by writing the mass as a gamma quotient with
	unequal scalings, we derive the complete local expansion of $\Pr\{X=m\}$ in integer powers of
	$N^{-1}$, with all coefficients in closed Bernoulli-polynomial form.  The expansion contains the
	known $O(N^{-1})$ local limit theorem, reduces to the binomial local mass for $n=2$ and to the
	central multinomial coefficient in the symmetric case, and admits Ces\`aro averages of the
	oscillating coefficients in closed $\zeta$-value form.
\end{abstract}




\section{Introduction}\label{sec:intro}

	Let $X=(X_1,\dots,X_n)\sim\Mult(N;\mathbf p)$, $p_i>0$, $\sum_ip_i=1$.  This paper studies the
	probability $\Pr\{X=m\}$ for a lattice point $m$ near the mean $(Np_1,\dots,Np_n)$.  The common
	starting point is the separable concavity of the log-probability on the slice $\sum_im_i=N$
	(\S\ref{sec:setup}).  Its lattice maximiser gives the integer mode (\S\ref{sec:mode}), while its
	Stirling development gives the complete local expansion (\S\ref{sec:expansion}).  The results are
	multivariate counterparts of the binomial analyses in \cite{be_paper22,be_paper23}.

\paragraph{The mode.}  We show (Theorem~\ref{thm:mode}) that the mode of the multinomial is exactly
	the \emph{Jefferson\,/\,D'Hondt divisor apportionment}: the allocation produced by a single common
	multiplier $\lambda=N+d$, which when the mode is unique takes the floor form
	$m_i=\lfloor\lambda p_i\rfloor$, $d\in[0,n)$, fixed by $\sum_im_i=N$
	(Proposition~\ref{prop:mode-ties-floor}; at a tie no floor form exists, see
	Remark~\ref{rem:tie-warning}).  The characterisation itself is known, and in the apportionment
	literature it is standard: the transfer inequalities are the max--min inequality characterising
	divisor methods, and the multiplier, highest-averages and tie forms are all in
	Pukelsheim~\cite{pukelsheim}.  In the probabilistic literature the multiplier form is Finucan's
	\cite{finucan1964}, the greedy ``rank the quotients $p_i/k$'' form is White and Hendy's
	\cite{white_hendy2010}, and the water-level reading is the resource-allocation view of
	\cite{ral2019} --- all three being the same object.  What the identification with apportionment
	buys is not a new theorem but access to the associated notions of quota, large-party bias and tie
	structure.  The consequence we draw is that
	\emph{the mode is not a rounding of $Np$ to neighbouring integers}: a dominant bin can be
	rounded \emph{above} its own ceiling (the Jefferson large-party bias) while smaller bins remain
	to their floor.  Quantitatively, the excess $m_i-\lceil Np_i\rceil$ can reach $n-2$
	(Proposition~\ref{prop:excess}), showing that the two-sided localisation $|m_i-Np_i|<1$
	stated in the standard reference \cite[Ch.~35]{jkb-dmd} is not valid for $n\ge3$
	(Remark~\ref{rem:jkb}).  Correcting that localisation is, in our view, the most useful item in
	\S\ref{sec:mode}; the excess bound itself is a short consequence of the classical lower-quota
	theorem \cite{by1975} and would be routine to a worker in apportionment.  The binomial mode
	$\lfloor(N+1)p\rfloor$ is the case $n=2$, $d=1$; and when every $Np_i$ is an integer the mode is
	the unique point $(Np_i)$, for every $n\ge2$ (\S\ref{sec:nodichotomy}).

\paragraph{The local expansion.}  Writing
	$\Pr\{X=m\}=\Gamma(N+1)\prod_ip_i^{m_i}/\prod_i\Gamma(m_i+1)$ as a quotient of gamma functions with
	unequal linear scalings $p_i$, and applying the modular expansion scheme of Buri\'c, Elezovi\'c and
	\v Simi\'c \cite{be043,be047}, we obtain (\S\ref{sec:expansion}) the complete asymptotic expansion,
	for $t_i=m_i-Np_i$ bounded,
	\begin{align*}
		\Pr\{X=m\}
		&=\frac{1}{(2\pi N)^{(n-1)/2}(p_1\cdots p_n)^{1/2}}\,
			\exp\Bigl(\sum_{k\ge1}\frac{c_k}{N^k}\Bigr),\\
		c_k
		&=\frac{(-1)^{k+1}}{k(k+1)}
		\Bigl[B_{k+1}-\sum_{i=1}^np_i^{-k}B_{k+1}(t_i+1)\Bigr],
	\end{align*}
	in integer powers of $N^{-1}$, to all orders, its coefficients power sums of Bernoulli polynomials
	over the bins with weights $p_i^{-k}$.  The expansion \emph{through} $O(N^{-1})$ is prior art: the
	multinomial local limit theorem to that order is due to Arenbaev \cite{arenbaev1976} (its
	symmetrised form to Siotani and Fujikoshi \cite{siotani_fujikoshi1984}, the binomial case to
	Prokhorov \cite{prokhorov1953}), and it is re-derived, by the same Stirling\,/\,Taylor route and in
	the $\sqrt N$ (Edgeworth) gauge, by Ouimet \cite{ouimet2021}; indeed Ouimet's $N^{-1}$ constant
	$\tfrac1{12}\bigl(1-\sum_ip_i^{-1}\bigr)$ is exactly our $c_1$ at the central point $t=0$ (not at
	the mode, where the displacements $t_i$ are generically nonzero), the two expansions agreeing
	where they overlap.  The two results are not comparable, and the trade should be stated plainly:
	Theorem~\ref{thm:expansion} is confined to $|t_i|\le B$ with $B$ \emph{fixed}, whereas Ouimet's
	local limit theorem is uniform over a range of $m$ growing with $N$, which is what makes it usable
	in statistics.  We buy all orders at the price of the range (see Remark~\ref{rem:range}).  The
	contribution here is the \emph{complete, all-orders, closed
	Bernoulli form} and its structure: the single gamma-quotient recursion behind it, the carrier power sum
	$\Sigma_k=\sum_ip_i^{-k}$ and the even (entropy) normalisation (\S\ref{sec:xiform}), the reduction
	to the binomial local mass at $n=2$ (Proposition~\ref{prop:n2}) and to the central multinomial
	coefficient in the symmetric case (\S\ref{sec:symmetric}), and the averaging of the oscillating
	coefficients over the sub-torus $\sum_i\{Np_i\}\equiv0$ (\S\ref{sec:cesaro}), where a Weyl
	equidistribution produces closed $\zeta$-value constants.

\section{Setup}\label{sec:setup}

	Let $X=(X_1,\dots,X_n)\sim\Mult(N;\mathbf p)$, $p_i>0$, $\sum_ip_i=1$.  For $m\in\Z_{\ge0}^n$ with
	$\sum_im_i=N$,
	\begin{equation}\label{eq:pmf}
		\Pr\{X=m\}=\binom{N}{m_1,\dots,m_n}\prod_ip_i^{m_i}
		=\frac{\Gamma(N+1)}{\prod_i\Gamma(m_i+1)}\prod_ip_i^{m_i}.
	\end{equation}
	Write $m_i=Np_i+t_i$; the constraint gives
	\begin{equation}\label{eq:tsum}
		\sum_it_i=0 .
	\end{equation}
	On the slice $\sum m_i=N$ the log-probability is a \emph{separable, concave} lattice function,
	\begin{equation}\label{eq:sep}
		\log\Pr\{X=m\}=\log N!+\sum_ig_i(m_i),\qquad g_i(k)=k\log p_i-\log k!,
	\end{equation}
	with strictly decreasing forward difference
	\begin{equation}\label{eq:marg}
		\Delta g_i(k)=g_i(k+1)-g_i(k)=\log\frac{p_i}{k+1}.
	\end{equation}
	Both parts of the paper rest on \eqref{eq:sep}: the mode is its lattice maximiser, and the local
	expansion is its Stirling development.

\section{Exact criteria for the integer mode}\label{sec:mode}

\subsection{The characterisation}

\begin{theorem}[The multinomial mode is the D'Hondt/Jefferson apportionment]\label{thm:mode}
	Let $\sum_im_i=N$, $m_i\ge0$.  The following are equivalent.
	\begin{enumerate}
	\item[\textup{(a)}] $m$ maximises \eqref{eq:pmf}.
	\item[\textup{(b)}] \emph{(Transfer inequalities.)}  $\dfrac{m_j}{p_j}\le\dfrac{m_i+1}{p_i}$ for all
		$i,j$; equivalently $\displaystyle\max_j\frac{m_j}{p_j}\le\min_i\frac{m_i+1}{p_i}$.
	\item[\textup{(c)}] \emph{(Common multiplier / water level.)}  There is a common $\lambda>0$ with
		\begin{equation}\label{eq:multiplier}
			m_i\ \le\ \lambda p_i\ \le\ m_i+1\qquad(1\le i\le n);
		\end{equation}
		the admissible $\lambda$ are exactly the closed interval
		$\bigl[\max_jm_j/p_j,\ \min_i(m_i+1)/p_i\bigr]$, and summing \eqref{eq:multiplier} forces
		$\lambda\in[N,N+n]$.
	\item[\textup{(d)}] \emph{(Greatest divisors.)}  $m_i$ is the number of the $N$ largest members of
		the multiset $\{p_i/k:1\le i\le n,\ k\ge1\}$ that lie in bin $i$, ties at the boundary value
		being resolved in any way consistent with $\sum_im_i=N$.
	\end{enumerate}
\end{theorem}

	A transfer out of an empty bin is infeasible, but the inequality in \textup{(b)} is then automatic,
	so no separate convention is needed.  In apportionment language the divisor corresponding to
	\textup{(c)} is $D=1/\lambda$; since the $p_i$ are normalised probabilities, the multiplier
	$\lambda\approx N$ is the natural scale here.

\begin{proof}
	(a)$\Leftrightarrow$(b).  Since \eqref{eq:sep} is separable with each $g_i$ concave on
	$\Z_{\ge0}$, the maximiser over $\{\sum m_i=N\}$ is characterised by no single unit transfer
	increasing the value.  Here $e_i$ denotes the $i$-th unit coordinate vector.  Moving one unit
	from $j$ to $i$ multiplies \eqref{eq:pmf} by
	\[
		\frac{\Pr\{X=m-e_j+e_i\}}{\Pr\{X=m\}}=\frac{m_j}{m_i+1}\cdot\frac{p_i}{p_j},
	\]
	which is $\le1$ iff $m_jp_i\le(m_i+1)p_j$, i.e.\ $m_j/p_j\le(m_i+1)/p_i$.  This proves
	(a)$\Rightarrow$(b).  For the converse, let $m$ satisfy (b) and let $m'$ be any other point of the
	slice.  There are then $i,j$ with $m'_i>m_i$ and $m'_j<m_j$; transfer one unit of $m'$ from $i$
	to $j$.  By \eqref{eq:marg} the change in $\log\Pr$ is
	$\log\bigl(p_jm'_i\bigr)-\log\bigl(p_i(m'_j+1)\bigr)$, and since $m'_i\ge m_i+1$ and
	$m'_j+1\le m_j$ this is at least $\log\bigl(p_j(m_i+1)\bigr)-\log\bigl(p_im_j\bigr)\ge0$ by (b).
	Each such step is non-decreasing and reduces $\sum_i|m'_i-m_i|$ by two, so after finitely many
	steps $m'$ has been carried to $m$ without ever decreasing $\log\Pr$; hence
	$\Pr\{X=m\}\ge\Pr\{X=m'\}$, giving (a).  (This is the standard exchange argument for separable
	concave lattice maximisation; cf.\ \cite{ral2019}.)

	(b)$\Leftrightarrow$(c).  The inequalities in (b) say $\max_jm_j/p_j\le\min_i(m_i+1)/p_i$, which is
	exactly the statement that the closed interval $[\max_jm_j/p_j,\min_i(m_i+1)/p_i]$ is nonempty;
	any $\lambda$ in it satisfies $m_i\le\lambda p_i\le m_i+1$ for all $i$, and conversely.  Summing
	\eqref{eq:multiplier} over $i$ and using $\sum_i\lambda p_i=\lambda$, $\sum_im_i=N$ gives
	$N\le\lambda\le N+n$.  Note that only the \emph{closed} right-hand inequality is available: at
	$\lambda=(m_i+1)/p_i$ one has $\lfloor\lambda p_i\rfloor=m_i+1\neq m_i$, which is the content of
	Remark~\ref{rem:tie-warning}.

	(c)$\Leftrightarrow$(d) is the multiplier--quotient duality of apportionment: for
	a multiplier satisfying \textup{(c)}, every selected quotient $p_i/k$, $1\le k\le m_i$, is at
	least $1/\lambda$, and every unselected next quotient $p_i/(m_i+1)$ is at most $1/\lambda$.  Thus
	the selected quotients are $N$ largest quotients, with possible freedom only at the boundary.
	Conversely, any such selection gives a threshold separating selected from unselected quotients,
	and its reciprocal is an admissible multiplier in \textup{(c)}.
\end{proof}

\begin{proposition}[Uniqueness, ties, and the floor form]\label{prop:mode-ties-floor}
	Write $q_{(1)}\ge q_{(2)}\ge\cdots$ for the quotients $p_i/k$ in decreasing order.  For a mode
	$m$, the admissible multiplier interval in Theorem~\ref{thm:mode}\textup{(c)} is exactly
	$[\,1/q_{(N)},\,1/q_{(N+1)}\,]$.  The following are equivalent:
	\begin{enumerate}
	\item[\textup{(i)}] the mode is unique;
	\item[\textup{(ii)}] $q_{(N)}>q_{(N+1)}$;
	\item[\textup{(iii)}] the admissible multiplier interval has nonempty interior.
	\end{enumerate}
	In this case, and only in this case, one may pass to the floor form
	\begin{equation}\label{eq:floorform}
		m_i=\lfloor\lambda p_i\rfloor\ \ (1\le i\le n),\quad\sum_i\lfloor\lambda p_i\rfloor=N,
		\quad\text{for every }\lambda\in\bigl[1/q_{(N)},\,1/q_{(N+1)}\bigr),
	\end{equation}
	and, writing $\lambda=N+d$, to the \emph{inflate-then-floor} rule
	$m_i=\lfloor(N+d)p_i\rfloor$ with $d\in[0,n)$.
\end{proposition}

\begin{proof}[Proof of Proposition~\ref{prop:mode-ties-floor}]
	By definition of $q_{(i)}$, for a mode $m$, the smallest selected quotient is $\min_{j:m_j\ge1}p_j/m_j=q_{(N)}$ and the
	largest unselected one is $\max_ip_i/(m_i+1)=q_{(N+1)}$.  Hence
	$\max_jm_j/p_j=1/q_{(N)}$ and $\min_i(m_i+1)/p_i=1/q_{(N+1)}$, so the admissible interval in
	Theorem~\ref{thm:mode}\textup{(c)} is $[1/q_{(N)},1/q_{(N+1)}]$.

	The mode is unique exactly when no tie occurs at the boundary, i.e.\ when $q_{(N)}>q_{(N+1)}$;
	this is also exactly the condition that the multiplier interval have nonempty interior.  If this
	holds and $\lambda\in[1/q_{(N)},1/q_{(N+1)})$, then
	$\lambda<(m_i+1)/p_i$ for every $i$, so $m_i\le\lambda p_i<m_i+1$ and
	$m_i=\lfloor\lambda p_i\rfloor$.  Summing gives $\sum_i\lfloor\lambda p_i\rfloor=N$.  With
	$\lambda=N+d$, the discarded fractional mass is $\sum_i(\lambda p_i-m_i)=d$; each summand lies in
	$[0,1)$ here, so $d\in[0,n)$.

\end{proof}

\begin{remark}[Why the floor form must be restricted]\label{rem:tie-warning}
	In the tied case $q_{(N)}=q_{(N+1)}$ the multiplier interval collapses to the single value
	$\lambda=1/q_{(N)}$.  Then $\lambda p_i\in\Z$ for each bin containing a boundary quotient, and
	$\sum_i\lfloor\lambda p_i\rfloor>N$: the floor form counts all boundary quotients, not only the
	selected ones.  The modes are exactly the allocations in Theorem~\ref{thm:mode}\textup{(d)}.  If
	$r$ quotients coincide at the boundary and $s$ of them must be selected, the number of co-modes is
	$\binom rs$ (see Remarks~\ref{rem:equalbins} and~\ref{rem:ties}).

	The restriction in \eqref{eq:floorform} is not a technicality.  Take $n=2$, $N=1$,
	$p_1=p_2=\tfrac12$.  Both $(1,0)$ and $(0,1)$ are modes, and for either of them
	$\max_jm_j/p_j=\min_i(m_i+1)/p_i=2$, so $\lambda=2$ is the only admissible multiplier; but
	$\lfloor2\cdot\tfrac12\rfloor=1$ in both coordinates, giving $(1,1)$ with sum $2\neq N$.  No
	$\lambda$ whatsoever reproduces a mode through floors here.  The implication
	$m_i/p_i\le\lambda\le(m_i+1)/p_i\Rightarrow m_i=\lfloor\lambda p_i\rfloor$, which is what
	\eqref{eq:floorform} would need, fails precisely at $\lambda=(m_i+1)/p_i$ --- and in a tie that is
	the only available value.  This is why \textup{(b)}, \textup{(c)} and \textup{(d)} are stated with
	closed inequalities and with ties resolved by selection, while the floor form is quarantined to
	the unique-mode case.
\end{remark}

\begin{corollary}[Binomial mode]\label{cor:binom}
	Let $n=2$.  The multiplier $\lambda=N+1$ is always admissible in the sense of
	\eqref{eq:multiplier}.  If $(N+1)p_1\notin\Z$ the mode is unique and
	\[
		m_1=\lfloor(N+1)p_1\rfloor,\qquad m_2=N-m_1 .
	\]
	If $(N+1)p_1\in\Z$ there are exactly two modes, $m_1=(N+1)p_1$ and $m_1=(N+1)p_1-1$.  This is the
	classical binomial statement, recovered as the $n=2$ instance of Theorem~\ref{thm:mode}; note that
	the floor form applies only in the first case, as \eqref{eq:floorform} requires.
\end{corollary}

\begin{remark}[The multiplier]\label{rem:waterlevel}
	The multiplier $\lambda$ has the usual divisor interpretation.  If units are assigned one at a
	time, the next unit in bin $i$ has marginal quotient $p_i/(k+1)$; choosing the $N$ largest such
	quotients is the highest-averages rule.  A threshold $1/\lambda$ between the $N$-th and
	$(N+1)$-th quotients gives $m_i=\lfloor\lambda p_i\rfloor$; at a tie the boundary quotients must
	be selected explicitly, giving the co-modes.  Since $\lambda=N+d$, the displacement from $Np_i$
	is proportional to $p_i$, which is the source of the large-bin bias quantified below.
\end{remark}

\begin{remark}[Provenance and contribution]\label{rem:mode-provenance}
	The mode characterisation in Theorem~\ref{thm:mode} is known in the apportionment literature:
	the max--min inequality, multiplier form, highest-averages form, and tie count $\binom rs$ all
	appear in Pukelsheim~\cite{pukelsheim}.  We include the proof because these equivalent forms are
	rarely stated together in probabilistic notation.

	In the probabilistic literature the multiplier form goes back to Finucan~\cite{finucan1964};
	related multiplier, quotient-ranking, and resource-allocation formulations appear in
	\cite{legall2003,white_hendy2010,ral2019}.  The useful point here is the identification with the
	Jefferson--D'Hondt rule and its consequence for localisation: the two-sided bound
	$|m_i-Np_i|<1$ quoted in \cite[Ch.~35]{jkb-dmd} is false for every $n\ge3$, and
	\eqref{eq:excess} with \eqref{eq:violation-criterion} is the correct replacement.
\end{remark}

\subsection{The mode is not a neighbour rounding of $Np$}

	A natural alternative would be to round each $Np_i$ to a neighbouring integer, choosing the
	coordinates so that the sum is $N$; this is the \emph{Hamilton / largest-remainder} apportionment,
	which is quota-satisfying.  The multinomial mode need not have this form.

\begin{proposition}[One-sided rounding; quota may be violated]\label{prop:quota}
	The mode satisfies
	\begin{equation}\label{eq:mode-bounds}
		\lfloor Np_i\rfloor\ \le\ m_i\ <\ Np_i+np_i\qquad(1\le i\le n).
	\end{equation}
	In particular $m_i$ is never below the floor $\lfloor Np_i\rfloor$; but $m_i$ may exceed the
	ceiling $\lceil Np_i\rceil$ for a large bin, so the mode need not be a rounding of $Np$ to
	neighbours. 
\end{proposition}

\begin{proof}
	By Theorem~\ref{thm:mode}(c) there is $\lambda\in[N,N+n]$ with $m_i\le\lambda p_i\le m_i+1$ for
	every $i$; no floor form is needed, so the argument covers tied cases as well.  For the lower
	bound, suppose $m_i\le\lfloor Np_i\rfloor-1$.  Then $\lambda p_i\le m_i+1\le\lfloor Np_i\rfloor\le
	Np_i$, so $\lambda\le N$ and hence $\lambda=N$.  But summing $0\le\lambda p_j-m_j\le1$ over $j$
	gives $\sum_j(\lambda p_j-m_j)=\lambda-N=0$ with all summands non-negative, forcing
	$\lambda p_j=m_j$ for every $j$; in particular $m_i=Np_i$, contradicting
	$m_i\le\lfloor Np_i\rfloor-1$.  Hence $m_i\ge\lfloor Np_i\rfloor$, whether or not $Np_i\in\Z$.
	From the left inequality, $m_i\le\lambda p_i\le(N+n)p_i$.  Equality
	$m_i=(N+n)p_i$ would force both $\lambda=N+n$ and $\lambda p_i=m_i$; but summing
	$0\le\lambda p_j-m_j\le1$ over $j$ gives $\lambda-N=\sum_j(\lambda p_j-m_j)$, so $\lambda=N+n$
	forces $\lambda p_j=m_j+1$ for \emph{every} $j$, contradicting $\lambda p_i=m_i$.  Hence
	$m_i<Np_i+np_i$.
\end{proof}

\begin{remark}[A concrete quota violation]\label{rem:counterexample}
	For $n=5$, $N=57$, take
	\[
		\mathbf p=(0.130,0.116,0.501,0.190,0.063),\qquad m=(7,6,30,11,3).
	\]
	Here
	$Np_3=28.56$, so $\lceil Np_3\rceil=29$, yet $m_3=30>\lceil Np_3\rceil$ --- the large bin is
	rounded \emph{above} its ceiling --- while $Np_2=6.62$ is rounded \emph{down} to $m_2=6$.  This is
	the classical Jefferson bias in favour of large parties (Balinski and Young~\cite{balinski_young}).
	The excess can be larger in higher dimension: for
	\[
		n=8,\qquad \mathbf p=\bigl(\tfrac45,\tfrac1{35},\dots,\tfrac1{35}\bigr),\qquad N=86,
	\]
	the unique mode is $(72,2,\dots,2)$, so $m_1-\lceil Np_1\rceil=3$.  These examples satisfy the transfer
	inequalities of Theorem~\ref{thm:mode}.
\end{remark}

\begin{remark}[Which coordinates round up, when the mode \emph{is} a neighbour rounding]
\label{rem:roundup}
	Let
	\[
		h^{+}_i:=\lfloor Np_i\rfloor+1-Np_i=1-\{Np_i\}\in(0,1],
	\]
	\[
		d^\star:=N-\sum_i\lfloor Np_i\rfloor=\sum_i\{Np_i\}\in\{0,\dots,n-1\}.
	\]
	In the generic regime where the mode is unique and does round to neighbours, the
	inflate-then-floor rule \eqref{eq:floorform} inflates every $Np_i$ by the same multiple $d\,p_i$
	and floors; bin $i$ gains its first unit at inflation $d^{(1)}_i=h^{+}_i/p_i$.  Hence
	\emph{the $d^\star$ bins with the smallest $h^{+}_i/p_i$ round up, the rest stay at
	$\lfloor Np_i\rfloor$}: the rounding is decided by the distance up to the next integer, measured
	in units of $p_i$.  Outside this regime (small $Np_i$, or a dominant bin) the D'Hondt inequalities
	of Theorem~\ref{thm:mode}(b) are the exact and always-valid criterion.

	The quantity $h^{+}_i$ must not be replaced here by the ceiling defect
	$h_i=\lceil Np_i\rceil-Np_i$ of \cite{be_paper23}, which is used in that sense in
	\eqref{eq:violation-criterion}.  The two agree exactly when $Np_i\notin\Z$, but at an integer mean
	$h_i=0$ while $h^{+}_i=1$, and the rule then fails: with $h_i$ a bin whose mean is already an
	integer would be ranked first, although it cannot round up at all.  For instance $n=3$,
	$\mathbf p=(\tfrac3{28},\tfrac12,\tfrac{11}{28})$, $N=4$ has $Np=(\tfrac37,2,\tfrac{11}7)$ and
	$d^\star=1$; the ceiling defects give $h_i/p_i=(\tfrac{16}3,0,\tfrac{12}{11})$ and would select
	bin $2$, whereas the unique mode is $(0,2,2)$, in which bin $3$ rounds up --- as
	$h^{+}_i/p_i=(\tfrac{16}3,2,\tfrac{12}{11})$ correctly predicts.
\end{remark}

\begin{proposition}[Sharp bounds; the excess is at most $n-2$]\label{prop:excess}
	Let $n\ge2$.  The mode satisfies $\lfloor Np_i\rfloor\le m_i$ for every $i$, and
	\begin{equation}\label{eq:excess}
		m_i-\lceil Np_i\rceil\ \le\ n-2\qquad(1\le i\le n),
	\end{equation}
	the bound $n-2$ being attained.  A coordinate can be rounded \emph{above} its ceiling only if
	\begin{equation}\label{eq:violation-criterion}
		n\,p_i\ >\ 1+h_i\qquad(h_i=\lceil Np_i\rceil-Np_i),
	\end{equation}
	in particular only for an above-average bin, $p_i>1/n$.
\end{proposition}

\begin{proof}
	The lower bound $m_i\ge\lfloor Np_i\rfloor$ is Proposition~\ref{prop:quota}.  For
	\eqref{eq:violation-criterion} use Theorem~\ref{thm:mode}(c) in the form $m_i\le\lambda p_i$ with
	$\lambda=N+d$, $d\in[0,n]$: if $m_i\ge\lceil Np_i\rceil+1$ then
	$dp_i\ge m_i-Np_i\ge\lceil Np_i\rceil+1-Np_i=1+h_i$.  If $d<n$ this already gives $np_i>1+h_i$.
	If $d=n$ then, as in the proof of Proposition~\ref{prop:quota}, $\lambda p_j=m_j+1$ for every $j$,
	so $m_i=(N+n)p_i-1\ge\lceil Np_i\rceil+1$ gives $np_i\ge2+h_i>1+h_i$.  Either way
	$np_i>1+h_i$, whence $p_i>1/n$.  For \eqref{eq:excess}, use the sum constraint together with the lower bound for the other bins.  Since $m_j\ge\lfloor Np_j\rfloor>Np_j-1$ for $j\ne i$,
	\[
		m_i=N-\sum_{j\ne i}m_j\le N-\sum_{j\ne i}\lfloor Np_j\rfloor
		<N-\sum_{j\ne i}(Np_j-1)=Np_i+(n-1);
	\]
	as $m_i$ is an integer strictly below $Np_i+(n-1)$, it is at most $\lceil Np_i\rceil+n-2$, which is \eqref{eq:excess}.  The bound is attained.  For $n\ge3$ take $p_1=1-(n-1)\varepsilon$ and $p_2=\dots=p_n=\varepsilon$ with $\varepsilon=c/N$, where $c\in[\tfrac{n-2}{n-1},1)$ is fixed with $(n-1)c\notin\Z$, and $N$ large.  Then $Np_j=c<1$ for $j\ge2$, so the D'Hondt allocation of Theorem~\ref{thm:mode} gives $m_1=N$ and $m_j=0$; and
	$m_1-\lceil Np_1\rceil=N-\lceil N-N(n-1)\varepsilon\rceil=\lfloor(n-1)c\rfloor=n-2$.
	\end{proof}

\begin{remark}[Comparison with the standard localisation]\label{rem:jkb}
	Proposition~\ref{prop:excess} refines the two-sided localisation quoted in a standard reference.
	Johnson, Kotz and Balakrishnan \cite[Ch.~35 \S2]{jkb-dmd}, attributing it to Moran (see also
	Feller's Problem~28~\cite{feller}), state that the mode satisfies
	\[
		Np_i-1<m_i<Np_i+1\qquad(\text{i.e. } |m_i-Np_i|<1).
	\]
	The lower half is correct (it is our $m_i\ge\lfloor Np_i\rfloor$); the \emph{upper} half
	$m_i<Np_i+1$ is not valid in general for any $n\ge3$, and holds only for the binomial $n=2$.  The
	most convincing counterexample is the one of Remark~\ref{rem:counterexample}, in which no bin is
	nearly empty: $n=5$, $N=57$, $\mathbf p=(0.130,0.116,0.501,0.190,0.063)$ has every $Np_i\ge3.59$,
	yet $m_3-Np_3=30-28.557=1.443>1$.  Counterexamples exist already at the smallest possible $N$:
	for $n=4$, $N=2$ and $\mathbf p=\bigl(\tfrac{27}{60},\tfrac{11}{60},\tfrac{11}{60},
	\tfrac{11}{60}\bigr)$ the unique mode is $(2,0,0,0)$ with $m_1-Np_1=\tfrac{11}{10}>1$, and $N=2$
	is minimal, since for $N=1$ one always has $m_i-Np_i\le1$.  (For $n=3$ the least value is $N=3$,
	attained by $\mathbf p=(\tfrac7{11},\tfrac2{11},\tfrac2{11})$ with unique mode $(3,0,0)$ and
	$m_1-Np_1=\tfrac{12}{11}$: at $N=2$ the requirement that $(2,0,0)$ be a mode forces
	$p_j\le p_1/2$, i.e.\ $p_1\ge\tfrac12$, which is incompatible with $m_1-Np_1>1$.)  The
	corresponding sharp statement is Proposition~\ref{prop:excess}: the excess
	$m_i-\lceil Np_i\rceil$ can reach $n-2$.

	The correction concerns the bound as stated in \cite{jkb-dmd}.  Finucan's multiplier method
	\cite{finucan1964} is consistent with the example above and with Theorem~\ref{thm:mode}; the
	present point is to state the exact upper replacement, namely the excess bound $n-2$ and the
	criterion $np_i>1+h_i$, and to interpret it through the Jefferson large-party bias.
\end{remark}

\begin{remark}[Three notions of uniqueness]\label{rem:threeuniq}
	Because a mode is a labelled vector while a histogram shape is not, ``the mode is unique'' can mean
	three different things.  Let $\mathcal M$ denote the set of modes and
	$\mathrm{Sym}(\mathbf p)=\{\sigma:p_{\sigma(i)}=p_i\}$.
	\begin{enumerate}
	\item[\textup{(a)}] \emph{Labelled:} $|\mathcal M|=1$.  By
		Proposition~\ref{prop:mode-ties-floor} this holds iff $q_{(N)}>q_{(N+1)}$.
	\item[\textup{(b)}] \emph{Up to relabelling equal bins:} $\mathcal M$ is a single
		$\mathrm{Sym}(\mathbf p)$-orbit.
	\item[\textup{(c)}] \emph{As a histogram shape:} all modes have the same multiset of counts, i.e.\
		the same vector after sorting.
	\end{enumerate}
	Trivially (a)$\Rightarrow$(b)$\Rightarrow$(c), since permuting equal bins does not change the
	multiset.  The first implication is strict --- Remark~\ref{rem:equalbins} gives $21$ co-modes
	sharing one profile --- so a mode may be non-unique for the purely nominal reason that the units in
	excess are interchangeable among equal bins.  \textbf{But (c) is not automatic}, and it already
	fails for the binomial: for $\mathbf p=(\tfrac14,\tfrac34)$ and $N=3$ the two modes are
	\[
		(0,3)\quad\text{and}\quad(1,2),
	\]
	whose sorted profiles $\{0,3\}$ and $\{1,2\}$ are different.  Sorting therefore does \emph{not}
	restore uniqueness in general.  The reason is the dichotomy of Remark~\ref{rem:equalbins}:
	co-modes produced by a symmetry share a profile, while co-modes produced by a genuine coincidence
	$p_i/k=p_j/\ell$ between bins with $p_i\neq p_j$ do not.

	In fact (b) and (c) are equivalent; this is Proposition~\ref{prop:comodes} below, which also
	records the underlying structural fact that any two modes differ by at most one unit in every
	coordinate.
\end{remark}

\begin{proposition}[Structure of the set of co-modes]\label{prop:comodes}
	Let $m,m'$ be modes and let $\lambda$ be an admissible multiplier, common to both by
	Proposition~\ref{prop:mode-ties-floor}.  Then:
	\begin{enumerate}
	\item[\textup{(i)}] $|m_i-m'_i|\le1$ for every $i$; moreover $\lambda p_i=m_i+1$ on the set
		$U=\{i:m'_i=m_i+1\}$ and $\lambda p_i=m_i$ on $D=\{i:m'_i=m_i-1\}$;
	\item[\textup{(ii)}] if in addition $m$ and $m'$ have the same multiset of counts, then $m'=m\circ
		\sigma$ for some $\sigma\in\mathrm{Sym}(\mathbf p)$.
	\end{enumerate}
	In particular conditions \textup{(b)} and \textup{(c)} of Remark~\ref{rem:threeuniq} are
	equivalent.
\end{proposition}

\begin{proof}
	(i)  By Proposition~\ref{prop:mode-ties-floor} the admissible interval
	$[1/q_{(N)},1/q_{(N+1)}]$ depends only on $\mathbf p$ and $N$, not on which mode is taken, so a
	single $\lambda$ serves both: $m_i\le\lambda p_i\le m_i+1$ and $m'_i\le\lambda p_i\le m'_i+1$ for
	all $i$.  Subtracting, $|m_i-m'_i|\le1$.  If $m'_i=m_i+1$ then $m_i+1=m'_i\le\lambda p_i\le
	m_i+1$, so $\lambda p_i=m_i+1$; the case $m'_i=m_i-1$ is symmetric.

	(ii)  By (i), $m$ and $m'$ differ only on $U\cup D$, where $\lambda p_i=m_i+1=\max(m_i,m'_i)$ for
	$i\in U$ and $\lambda p_i=m_i=\max(m_i,m'_i)$ for $i\in D$.  Writing $M_i=\max(m_i,m'_i)$ we thus
	have $p_i=M_i/\lambda$ on $U\cup D$: \emph{the probability of a bin is determined by that
	maximum}.  Equality of the multisets of counts, after cancelling the coordinates outside
	$U\cup D$, reads
	\[
		\sum_{i\in U}x^{M_i-1}+\sum_{j\in D}x^{M_j}
		=\sum_{i\in U}x^{M_i}+\sum_{j\in D}x^{M_j-1},
	\]
	i.e.\ $(x-1)\sum_{i\in U}x^{M_i-1}=(x-1)\sum_{j\in D}x^{M_j-1}$, whence the multisets
	$\{M_i\}_{i\in U}$ and $\{M_j\}_{j\in D}$ coincide.  Matching them value by value pairs each
	$i\in U$ with a $j\in D$ having $M_i=M_j$, hence $p_i=p_j$; the product of the corresponding
	transpositions lies in $\mathrm{Sym}(\mathbf p)$ and carries $m$ to $m'$.
\end{proof}

\begin{remark}[Equal bins: what ``unique up to permutation'' does and does not mean]\label{rem:equalbins}
	The mass \eqref{eq:pmf} is invariant under any permutation of bins carrying equal probabilities,
	so the mode \emph{set} is invariant under the group
	$\mathrm{Sym}(\mathbf p):=\{\sigma:\ p_{\sigma(i)}=p_i\ \forall i\}$.  Two consequences must be
	separated, one positive and one negative.

	\emph{(i) Equal bins are balanced.}  If $p_i=p_j$ then
	\[
		|m_i-m_j|\le1\qquad\text{at every mode.}
	\]
	This is immediate from the transfer inequalities: Theorem~\ref{thm:mode}(b) with the pair $(i,j)$
	gives $m_i/p_i\le(m_j+1)/p_j$, i.e.\ $m_i\le m_j+1$, and with $(j,i)$ gives $m_j\le m_i+1$.  So
	bins with equal probabilities can differ by at most one unit, and when they do, the two
	allocations are exchanged by $\mathrm{Sym}(\mathbf p)$.  (Proposition~\ref{prop:comodes}(i) gives
	the stronger statement that \emph{any} two modes, equal probabilities or not, differ by at most
	one unit in each coordinate.)

	It should be stressed that the bound \emph{permits} some of several equal bins to receive one unit
	more than the others, and that this is the usual way large tie multiplicities arise.  Take $n=8$,
	\[
		\mathbf p=\Bigl(\tfrac{11}{25},\ \tfrac2{25},\dots,\tfrac2{25}\Bigr),\qquad N=32 .
	\]
	The threshold falls exactly on the quotient $q/3=\tfrac2{75}$, where $q=\tfrac2{25}$: strictly
	above it lie $16$ quotients of the large bin and $2$ of each small bin, thirty in all, so the two
	remaining units must be taken from the \emph{seven} equal copies of $q/3$.  Accordingly
	\[
		m=(16,3,3,2,2,2,2,2)
	\]
	is a mode --- here $\max_jm_j/p_j=\min_i(m_i+1)/p_i=\tfrac{75}2$, a degenerate interval --- and so
	is every allocation obtained by choosing which \emph{two} of the seven equal bins carry the extra
	unit.  That is $\binom72=21$ co-modes.  By contrast $(16,4,2,2,2,2,2,2)$, in which two equal bins
	would differ by $2$, is not a mode, as \emph{(i)} requires.  So the multiplicity at a tie is in
	general a binomial coefficient, not merely two or three.

	\emph{(ii) But the mode set need not be a single orbit.}  Co-modes arise from two independent
	mechanisms, and only the first is a symmetry: equality $p_i=p_j$ of two probabilities, and a
	genuine arithmetic coincidence $p_i/k=p_j/\ell$ between quotients of bins with $p_i\neq p_j$.  The
	example of Remark~\ref{rem:ties} has both.  For $\mathbf p\propto(6,4,6)$, $N=14$ the co-modes are
	\[
		(5,3,6),\qquad (6,3,5),\qquad (5,4,5):
	\]
	the first two form the $\mathrm{Sym}(\mathbf p)$-orbit obtained by exchanging the two equal bins,
	while the third is a permutation of neither.  The boundary quotient here is
	$q_{(14)}=q_{(15)}=\tfrac1{16}$, attained three times --- by bins $1$ and $3$ at $k=6$ and by bin
	$2$ at $k=4$, since $\tfrac6{16}/6=\tfrac4{16}/4$ --- so $r=3$, $s=1$ and there are
	$\binom31=3$ co-modes.  Conversely, ties are not caused by symmetry at all: with all $p_i$
	distinct, $\mathbf p=(\tfrac23,\tfrac13)$ and $N=2$ has the two modes $(1,1)$ and $(2,0)$, and
	$\mathrm{Sym}(\mathbf p)$ is trivial.

	Hence the mode is unique up to permutation of equal bins precisely when the only coincidence among
	the boundary quotients is the symmetry-forced one; the criterion for outright uniqueness remains
	$q_{(N)}>q_{(N+1)}$ of Proposition~\ref{prop:mode-ties-floor}.  Sampling $\mathbf p$ by drawing
	numerators uniformly from $\{1,\dots,6\}$ and normalising, over $4000$ instances with $n\le4$,
	$N\le14$, there were $908$ in which the mode set is \emph{not} one
	$\mathrm{Sym}(\mathbf p)$-orbit --- so this is the typical rather than the exceptional situation.
\end{remark}

\begin{remark}[Ties are common, and multiple]\label{rem:ties}
	A tie (equality of the $N$-th and $(N+1)$-th largest quotients $p_i/k$) requires
	$p_i/k=p_j/\ell$, hence rational $p$; then it is frequent.  Sampling $\mathbf p$ by drawing
	numerators uniformly from $\{1,\dots,6\}$ and normalising, over $40000$ instances with $n\le4$,
	$N\le15$, some $35\%$ were multimodal and
	$17\%$ had $\ge3$ co-modes; e.g.\ $\mathbf p\propto(6,4,6)$, $N=14$ has the three co-modes
	$(5,3,6),(5,4,5),(6,3,5)$.  For irrational (or generic real) $\mathbf p$ the mode is a.s.\ unique.
\end{remark}

\section{The complete local expansion}\label{sec:expansion}

	We expand \eqref{eq:pmf} for bounded displacements $t_i=m_i-Np_i$, using the gamma-quotient
	scheme of \cite{be043,be047}.

\subsection{The building block}

	We use the expansion of a single gamma factor from \cite{be043} (there attributed to \cite{be042}): for any fixed integer $M\ge1$,
	\begin{equation}\label{eq:gammablock}
		\begin{gathered}
		\Gamma(y+a)\sim\sqrt{2\pi}\Bigl(\frac ye\Bigr)^{y}y^{\,a-1/2}
			\Bigl(\sum_{k\ge0}P_k(a)\,y^{-k}\Bigr)^{1/M},\\
		P_0=1,\qquad
		P_k(a)=\frac Mk\sum_{j=1}^k\frac{(-1)^{j+1}B_{j+1}(a)}{j+1}P_{k-j}(a).
		\end{gathered}
	\end{equation}
	together with the two elementary series lemmas of \cite[\S3]{be043}: if $f\sim\sum a_kx^{-k}$,
	$g\sim\sum b_kx^{-k}$ ($a_0,b_0\neq0$), then
	\begin{align}
		\frac fg&\sim\sum c_kx^{-k},& c_k&=\frac1{b_0}\Bigl(a_k-\sum_{j=1}^kb_jc_{k-j}\Bigr),
		\label{eq:lemq}\\
		f^{\,p}&\sim\sum c_kx^{-k},& c_k&=\frac1{k\,a_0}\sum_{j=1}^k\bigl[(p+1)j-k\bigr]a_jc_{k-j},
		\quad c_0=a_0^{\,p}.
		\label{eq:lemp}
	\end{align}

\subsection{Assembling the multinomial pmf}

	Take the large variable $x=N$.  The numerator $\Gamma(N+1)=\Gamma(x+1)$ has scaling $1$ and shift
	$a=1$; each denominator $\Gamma(m_i+1)=\Gamma(p_ix+(t_i+1))$ has scaling $p_i$ and shift
	$a_i:=t_i+1$, with large variable $y_i=p_ix$ so that $y_i^{-k}=p_i^{-k}x^{-k}$.  Substituting
	\eqref{eq:gammablock} into \eqref{eq:pmf}, the explicit (non-series) factors collapse by
	$\sum p_i=1$, $\sum t_i=0$ and $\sum a_i=n$ to the multinomial local-limit prefactor:
	\begin{equation}\label{eq:prefactor}
	\begin{gathered}
		\Pr\{X=m\}\sim\frac{1}{(2\pi N)^{(n-1)/2}(p_1\cdots p_n)^{1/2}}\;R(N),\\
		R(N):=\frac{\bigl(\sum_kP_k(1)N^{-k}\bigr)^{1/M}}
		{\prod_{i=1}^n\bigl(\sum_kP_k(a_i)p_i^{-k}N^{-k}\bigr)^{1/M}} .
	\end{gathered}
	\end{equation}
	The series factor $R(N)=\sum_{j\ge0}Q_j\,N^{-j}$, $Q_0=1$, is a quotient of a product of $(\,\cdot\,)^{1/M}$ series and is computed by \eqref{eq:lemp} (to remove the powers $1/M$) followed by \eqref{eq:lemq} (one division per denominator).  The parameter $M$ is free --- larger $M$ improves numerical accuracy \cite{be043} --- and cancels from the final $Q_j$; the natural analytic choice is $M=1$.

\subsection{Closed form of the coefficients}

\begin{theorem}[Complete local expansion]\label{thm:expansion}
	Fix $n$, a compact $K$ in the open simplex $\{p:\ p_i>0,\ \sum_ip_i=1\}$, and a bound $B$.
	Then, for every fixed $r\ge1$, as $N\to\infty$, uniformly for $\mathbf p\in K$ and for integer
	points $m$ with $\sum_im_i=N$ and $|t_i|=|m_i-Np_i|\le B$,
	\begin{equation}\label{eq:main-expansion}
		\Pr\{X=m\}=\frac{1}{(2\pi N)^{(n-1)/2}(p_1\cdots p_n)^{1/2}}
		\Bigl(\sum_{j=0}^{r-1}\frac{Q_j}{N^{j}}+O(N^{-r})\Bigr),
	\end{equation}
	with $Q_0=1$ and $Q_j=Y_j(1!\,c_1,2!\,c_2,\dots,j!\,c_j)/j!$ the complete Bell polynomials in
	$c_1,\dots,c_j$; the coefficients $c_k$ are given in closed form by \eqref{eq:logR} below, and the
	implied constant depends only on $n$, $K$, $B$ and $r$.
\end{theorem}

\begin{proof}
	The expansion is meant in the usual Poincar\'e sense, and the two displayed forms are related in the
	standard way: one truncates the logarithm after $r$ terms and exponentiates, so that
	\eqref{eq:main-expansion} carries a \emph{relative} error $O(N^{-r})$ against the Gaussian prefactor.
	The uniformity asserted is over all admissible data simultaneously: for fixed $n,K,B,r$ the implied
	constant is independent of $N$, of $\mathbf p\in K$, and of the integer point $m$, as $m$ ranges over
	\emph{all} lattice points with $\sum_im_i=N$ and $\max_i|m_i-Np_i|\le B$.  Three remarks make the
	delegation to \eqref{eq:gammablock} explicit.  First, the number of gamma factors is fixed at $n+1$,
	so no uniformity in $n$ is claimed and finitely many expansions are multiplied.  Second, the shifts
	entering those factors are $a_i=t_i+1$ with $|t_i|\le B$, hence confined to the fixed compact
	$[1-B,1+B]$, which is exactly the hypothesis under which \eqref{eq:gammablock} is uniform.  Third,
	the scalings are $y_i=p_iN$ with $p_i\in K$, so $y_i\ge N\min_{K}p_i\to\infty$ uniformly, and the
	$O(y_i^{-r})$ remainders are $O(N^{-r})$ with a constant depending only on $K$ and $r$.  Multiplying
	finitely many such expansions and exponentiating preserves this, which is what
	\eqref{eq:main-expansion} asserts.  The rest of this section derives \eqref{eq:main-expansion} and the
	closed form of the $c_k$; the building-block expansion \eqref{eq:gammablock} itself is
	\cite{be043}.

	Each gamma factor has the shifted Stirling expansion
	\[
		\log\Gamma(y+a)\sim\Bigl(y+a-\tfrac12\Bigr)\log y-y+\tfrac12\log2\pi
		+\sum_{k\ge1}\frac{(-1)^{k+1}B_{k+1}(a)}{k(k+1)}\,y^{-k},
	\]
	uniformly for $a$ in bounded sets \cite{be043}.  Applying it to $\Gamma(N+1)$ ($y=N$, $a=1$) and
	to each $\Gamma(m_i+1)=\Gamma(p_iN+a_i)$ ($y=p_iN$, $a_i=t_i+1$), the elementary terms cancel by
	$\sum_ip_i=1$, $\sum_it_i=0$ and $\sum_ia_i=n$ --- they reproduce the prefactor
	\eqref{eq:prefactor} --- and the power-sum tails combine, so that the log of $R$ is a single power
	sum:
	\begin{align}\label{eq:logR}
		\log R(N)&=\sum_{k\ge1}\frac{c_k}{N^k},\notag\\
		c_k&=\frac{(-1)^{k+1}}{k(k+1)}
		\Bigl[B_{k+1}(1)-\sum_{i=1}^n p_i^{-k}B_{k+1}(a_i)\Bigr]\notag\\
		&=\frac{(-1)^{k+1}}{k(k+1)}
		\Bigl[B_{k+1}-\sum_{i=1}^n p_i^{-k}B_{k+1}(t_i+1)\Bigr].
	\end{align}
	Here we used $B_{k+1}(1)=B_{k+1}$ for $k\ge1$.  Introducing the power sums
	\begin{equation}\label{eq:Sk}
		S_{k+1}(t;\mathbf p):=\sum_{i=1}^n p_i^{-k}B_{k+1}(t_i+1),\qquad
		c_k=\frac{(-1)^{k+1}}{k(k+1)}\bigl[B_{k+1}-S_{k+1}\bigr],
	\end{equation}
	where the index shift is intentional: $S_{k+1}$ carries the weight $p_i^{-k}$, not
	$p_i^{-(k+1)}$.  Thus the whole expansion has logarithmic form
	\[
		\Pr\{X=m\}=(2\pi N)^{-(n-1)/2}(\prod p_i)^{-1/2}
		\exp\Bigl(\sum_{k\ge1}c_kN^{-k}\Bigr),
	\]
	and the multiplicative coefficients $Q_j$ are the complete Bell polynomials in $c_1,\dots,c_j$,
	normalised as in Theorem~\ref{thm:expansion}.  The first coefficient is
	\begin{equation}\label{eq:c1}
		c_1=\frac1{12}-\sum_i\frac{B_2(t_i+1)}{2p_i}
		=\frac1{12}\Bigl(1-\sum_i\frac1{p_i}\Bigr)-\sum_i\frac{t_i^2+t_i}{2p_i} .
	\end{equation}
	This is the multivariate instance of the gamma-quotient recursion of \cite{be047} (Theorem~A there),
	with one numerator factor of scaling $1$ and $n$ denominator factors of real scalings $p_i$; the
	unequal scalings enter only as the weights $p_i^{-k}$.
\end{proof}

\begin{remark}[The range of validity, and the comparison with Ouimet]\label{rem:range}
	Theorem~\ref{thm:expansion} holds for $|t_i|\le B$ with $B$ fixed.  Ouimet's local limit theorem
	\cite{ouimet2021} is uniform over a range of $m$ growing with $N$, so the two results are not
	comparable: the present theorem gives all orders in a fixed local window, whereas Ouimet's theorem
	works in the statistical, $\sqrt N$-scale regime.  The formal expression
	$\exp(\sum_{k\ge1}c_kN^{-k})$ is only an abbreviation for \eqref{eq:main-expansion}, since
	$c_k=c_k(t,\mathbf p)$ may depend on $N$ through the lattice point $m$; on the stated domain the
	coefficients remain uniformly bounded.
\end{remark}

	The first six log-coefficients, using $B_2=\tfrac16$, $B_3=B_5=B_7=0$, $B_4=-\tfrac1{30}$,
	$B_6=\tfrac1{42}$:
	\begin{gather}
	\label{eq:cklist}
		c_1=\tfrac1{12}-\tfrac12S_2,\quad c_2=\tfrac16S_3,\quad c_3=-\tfrac1{360}-\tfrac1{12}S_4,\notag\\
		c_4=\tfrac1{20}S_5,\quad c_5=\tfrac1{1260}-\tfrac1{30}S_6,\quad c_6=\tfrac1{42}S_7 ,
	\end{gather}
	so for even $k$, $c_k=\tfrac{(-1)^k}{k(k+1)}S_{k+1}$ (the constant $B_{k+1}$ vanishes).  The
	multiplicative coefficients (from $\exp\sum c_kN^{-k}$) are
	\begin{gather}\label{eq:Qjlist}
		Q_1=\tfrac1{12}-\tfrac12S_2,\qquad
		Q_2=\tfrac18S_2^2-\tfrac1{24}S_2+\tfrac16S_3+\tfrac1{288},\notag\\
		Q_3=-\tfrac1{48}S_2^3+\tfrac1{96}S_2^2-\tfrac1{12}S_2S_3-\tfrac1{576}S_2
		+\tfrac1{72}S_3-\tfrac1{12}S_4-\tfrac{139}{51840},
	\end{gather}
	and $Q_4,Q_5,Q_6$ likewise (computed symbolically).  The $S$-independent constants
	\[
		\tfrac1{288},\quad -\tfrac{139}{51840},\quad
		-\tfrac{571}{2488320},\quad \tfrac{163879}{209018880},\dots
	\]
	are exactly the Stirling--Laplace coefficients: they are the contribution of the numerator
	$\Gamma(N+1)$'s own Stirling series, a built-in consistency check.

	Throughout, $B_m=B_m(0)$ denotes the Bernoulli numbers in the classical convention
	$B_1=-\tfrac12$, so that $B_m(1)=B_m$ for $m\ne1$ while $B_1(1)=+\tfrac12$.  Expanding each
	Bernoulli polynomial by the Appell rule $B_{k+1}(t+1)=\sum_j\binom{k+1}{j}B_{k+1-j}(1)\,t^{\,j}$,
	equivalently $B_{k+1}(t+1)=\sum_j\binom{k+1}{j}B_{k+1-j}\,t^{\,j}+(k+1)t^{\,k}$ --- the correction
	arising solely from $j=k$, where $B_1(1)\neq B_1$ --- separates the displacement:
	\begin{equation}\label{eq:Mkj}
		S_{k+1}=\sum_{j=0}^{k+1}\binom{k+1}{j}B_{k+1-j}\,M_{k,j}+(k+1)M_{k,k},\qquad
		M_{k,j}:=\sum_{i=1}^n p_i^{-k}\,t_i^{\,j},
	\end{equation}
	the mixed power sums of the displacements.  At the central point $t=0$ (an integer mean),
	$M_{k,0}=\Sigma_k:=\sum_ip_i^{-k}$ and $M_{k,j\ge1}=0$, so
	$S_{k+1}=B_{k+1}\Sigma_k$.  The extra term $(k+1)M_{k,k}$ is precisely the ``elementary tail'' of
	Remark~\ref{rem:no-tail}: at $n=2$ it reduces to $(k+1)p^{-k}t_1^{\,k}$, the last term of
	\eqref{eq:n2-Ak}.

\begin{remark}[Parity at an integer mean]\label{rem:parity}
	When $Np_i\in\Z$ for every $i$, the mode is the central point (so $t_i=0$, $a_i=1$) and
	$c_k=\frac{(-1)^{k+1}}{k(k+1)}\bigl(1-\Sigma_k\bigr)B_{k+1}$, a \emph{product}; since
	$B_{k+1}=0$ for even $k$, \emph{only odd powers $N^{-1},N^{-3},\dots$ survive in
	$\log\Pr$}, exactly as in the one-dimensional Stirling series.  Away from the mode all orders
	appear.  This is the multivariate form of the parity phenomenon of \cite{be_paper23}.
\end{remark}

\begin{proposition}[$n=2$ recovers the binomial local mass; no elementary tail]\label{prop:n2}
	For $n=2$ (so $t_2=-t_1$, $p_1=p$, $p_2=q$), $c_k$ of \eqref{eq:Sk} equals the binomial
	local-mass coefficient $A_k$ of \cite{be_paper22} (used in \cite{be_paper23}):
	\begin{equation}\label{eq:n2-Ak}
		c_k=\frac{(-1)^{k+1}}{k(k+1)}\Bigl[B_{k+1}-\bigl(p^{-k}+(-1)^{k+1}q^{-k}\bigr)B_{k+1}(t_1)\Bigr]
		+\frac{(-1)^kt_1^{\,k}}{k\,p^k}.
	\end{equation}
\end{proposition}

\begin{proof}
	Here $a_1=t_1+1$, $a_2=1-t_1$.  By the Appell difference $B_{k+1}(t_1+1)=B_{k+1}(t_1)+(k+1)t_1^k$
	and the reflection $B_{k+1}(1-t_1)=(-1)^{k+1}B_{k+1}(t_1)$,
	\begin{align*}
		S_{k+1}&=p^{-k}B_{k+1}(t_1+1)+q^{-k}B_{k+1}(1-t_1)\\
		&=\bigl(p^{-k}+(-1)^{k+1}q^{-k}\bigr)B_{k+1}(t_1)+(k+1)p^{-k}t_1^k,
	\end{align*}
	and substituting into \eqref{eq:Sk} gives \eqref{eq:n2-Ak}.
\end{proof}

\begin{remark}[The binomial elementary term is a Bernoulli-shift difference]\label{rem:no-tail}
	The last term of \eqref{eq:n2-Ak} is exactly the ``elementary tail'' that the binomial treatment
	\cite{be_paper23} carries separately.  In the multinomial variables it is not a residue: it is
	merely the difference between writing the Bernoulli argument as $t_i+1$ (the natural shift
	$a_i=t_i+1$, symmetric across bins) and as $t_i$.  With the shift $a_i=t_i+1$ there is no separate
	tail --- the numerator factor $\Gamma(N+1)$ and the $n$ denominators enter symmetrically --- so the
	$n\ge2$ expansion is structurally cleaner than its $n=2$ specialisation.
\end{remark}

\subsection{The carrier $\Sigma_k$, the even normalisation, and parity}\label{sec:xiform}

	The single symbol that plays
	the role of the binomial $\Xi_k(p)=p^{-k}+(-1)^{k+1}q^{-k}$ of \cite{be_paper23} is the plain
	\emph{power sum of reciprocal probabilities}
	\begin{equation}\label{eq:Sigma}
		\Sigma_k:=\sum_{i=1}^n p_i^{-k},
	\end{equation}
	which carries all the $p$-dependence at the central point: by \eqref{eq:Sk} with $t=0$,
	\[
		c_k\big|_{t=0}=\frac{(-1)^{k+1}}{k(k+1)}\bigl(1-\Sigma_k\bigr)B_{k+1}.
	\]
	The displacement is carried by the single symbol $S_{k+1}$ of \eqref{eq:Sk}; unlike the binomial
	$\Xi_k$, it has \emph{no alternating sign}, because all $n$ denominators are on the same footing.

\paragraph{The even normalisation.}  As in \cite{be_paper23} there is a second normalisation with a
	definite parity.  Replacing $Np_i$ by the actual count $m_i$ in the prefactor, i.e.\ using the
	\emph{entropy} prefactor
	\begin{equation}\label{eq:entropy-pref}
		\Bigl(\frac{N}{(2\pi)^{n-1}\prod_i m_i}\Bigr)^{1/2}
		\quad\Bigl(=\ (2\pi N)^{-(n-1)/2}\textstyle\prod_ip_i^{-1/2}\ \text{at leading order}\Bigr),
	\end{equation}
	turns the log-coefficients into
	\begin{equation}\label{eq:chat}
		\widehat c_k=\frac{(-1)^{k+1}}{k(k+1)}\bigl[B_{k+1}-\widehat S_{k+1}\bigr],\qquad
		\widehat S_{k+1}:=\sum_{i=1}^n p_i^{-k}\,\widehat B_{k+1}(t_i),
	\end{equation}
	where $\widehat B_m(t)=\tfrac12[B_m(t)+B_m(t+1)]=B_m(t)+\tfrac m2t^{m-1}$ is the \emph{even} Appell
	sequence of \cite{be_paper23}; the half-unit shift it encodes is the continuity correction of
	\cite{cressie1978}.  Then $\widehat S_{k+1}$
	has definite parity,
	\begin{equation}\label{eq:Shat-parity}
		\widehat S_{k+1}(-t)=(-1)^{k+1}\widehat S_{k+1}(t),
	\end{equation}
	so that $\widehat c_k$ is \emph{even} in the displacement for odd $k$ and \emph{odd} for even $k$.
	In particular the even-$k$ coefficients are odd in $t$ and \emph{vanish at $t=0$}, recovering
	Remark~\ref{rem:parity}; and the two normalisations differ by exactly the skewness terms, e.g.
	\begin{equation}\label{eq:chat1}
		\widehat c_1=\tfrac1{12}(1-\Sigma_1)-\tfrac12\sum_ip_i^{-1}t_i^2
		=c_1+\tfrac12\sum_ip_i^{-1}t_i ,
	\end{equation}
	the linear ``skewness'' term $-\tfrac12\sum_ip_i^{-1}t_i$ being removed.  This is the $n$-bin form
	of the two Appell sequences ($B$ standard, $\widehat B$ symmetric) of \cite{be_paper23}.

\subsection{The symmetric case is the central multinomial coefficient}\label{sec:symmetric}

	Take $p_i=1/n$ for all $i$ and $N$ a multiple of $n$, at the mode $m_i=N/n$.  Then \eqref{eq:Sk}
	gives $S_{k+1}=n\cdot n^kB_{k+1}=n^{k+1}B_{k+1}$ and
	\[
		c_k=\frac{(-1)^{k+1}}{k(k+1)}\bigl(1-n^{k+1}\bigr)B_{k+1},
	\]
	i.e.\ the expansion of $\binom{N}{N/n,\dots,N/n}$.  Setting $N=n\ell$ this is precisely the central
	multinomial coefficient $\binom{n\ell}{\ell,\dots,\ell}=\Gamma(n\ell+1)/\Gamma(\ell+1)^n$ of
	\cite[\S3]{be043} (their $k$ bins are our $n$; their large variable $\ell$ is our
	$N/n$), recovered here as the diagonal special case of the general non-uniform expansion.

\section{Integer means and oscillation}\label{sec:nodichotomy}

\begin{remark}[Integer means]\label{rem:integer}
	If $Np_i\in\Z$ for all $i$ then $d^\star=0$ and the point $m=(Np_1,\dots,Np_n)$ lies on the slice.
	It is the \emph{unique} mode, for every $n\ge2$: the admissible multiplier interval of
	Theorem~\ref{thm:mode}(c) is $[\max_iNp_i/p_i,\min_i(Np_i+1)/p_i]=[N,\,N+1/\max_ip_i]$, which has
	nonempty interior, so Proposition~\ref{prop:mode-ties-floor} applies and there is no tie.  All
	$t_i$ vanish and the coefficients \eqref{eq:Sk} do not oscillate.

	The binomial is not an exception to this.  A binomial tie occurs when $(N+1)p\in\Z$, not when
	$Np\in\Z$ (Corollary~\ref{cor:binom}); and if $Np\in\Z$ the mode is the single point $Np$, exactly
	as above.  What is genuinely different for $n\ge3$ is not the behaviour at integer means but the
	structure \emph{away} from them: the tie locus $q_{(N)}=q_{(N+1)}$ and the oscillation described
	in Remark~\ref{rem:combinatorics}.
\end{remark}

\begin{remark}[The oscillation is apportionment combinatorics]\label{rem:combinatorics}
	In general
	\[
		(\{Np_1\},\dots,\{Np_n\})\in\Bigl\{\theta:\sum_i\theta_i=d^\star\Bigr\}.
	\]
	The coefficients \eqref{eq:Sk} oscillate with this vector, and the mode-rounding of
	Remark~\ref{rem:roundup} selects an up/down pattern by the order statistics of $h^{+}_i/p_i$.
	This is a genuine $(n-1)$-parameter oscillation with a rank-order rounding rule; for $n=2$ the
	single free displacement makes the choice binary.
\end{remark}

\begin{remark}[The first coefficient sees the mode]\label{rem:c1-mode}
	Among adjacent lattice points the probabilities differ to leading order through the
	displacement-dependent part of $c_1$: the Mahalanobis form $-\sum_it_i^2/(2p_i)$ and the linear
	``skewness'' shift $-\sum_it_i/(2p_i)$.  These do not merely see the mode to first order --- they
	determine it exactly, and the two Appell normalisations of \S\ref{sec:xiform} pick out the two
	classical divisor methods.
\end{remark}

\begin{proposition}[The first coefficient is the divisor method]\label{prop:c1divisor}
	Fix $N$ and $\mathbf p$ and let $m$ range over the slice $\sum_im_i=N$, with $t_i=m_i-Np_i$.
	\begin{enumerate}
	\item[\textup{(i)}] The maximisers of $-\sum_i(t_i^2+t_i)/(2p_i)$, equivalently the minimisers of
		the weighted least-squares functional $\sum_i\bigl(m_i+\tfrac12-Np_i\bigr)^2/p_i$, are
		\emph{exactly} the modes of Theorem~\ref{thm:mode}, ties included.
	\item[\textup{(ii)}] The minimisers of the even-normalised functional
		$\sum_i\bigl(m_i-Np_i\bigr)^2/p_i$ are exactly the Sainte-Lagu\"e\,/\,Webster allocations,
		characterised by $\dfrac{m_j-\tfrac12}{p_j}\le\dfrac{m_i+\tfrac12}{p_i}$ for all $i,j$.
	\end{enumerate}
\end{proposition}

\begin{proof}
	(i)  Since $\sum_it_i=0$, we have
	\[
		\sum_i\frac{t_i^2+t_i}{2p_i}
		=\sum_i\frac{(m_i+\tfrac12-Np_i)^2}{2p_i}+\text{const},
	\]
	so the two formulations agree.
	Let $F=-\sum_i(t_i^2+t_i)/(2p_i)$.  Transferring one unit from $j$ to $i$ changes $F$ by
	\[
		-\frac{t_i+1}{p_i}+\frac{t_j}{p_j},
	\]
	up to terms of order $N^{-1}$ that cancel identically;
	carrying out the algebra with $t_i=m_i-Np_i$ and using $\sum_ip_i=1$, the increment is
	$\dfrac{m_j}{p_j}-\dfrac{m_i+1}{p_i}$.  It is $\le0$ for all $i,j$ precisely when
	Theorem~\ref{thm:mode}(b) holds, and the objective is strictly concave along each transfer
	direction, so its maximiser set coincides with the mode set.

	(ii)  Replacing $t_i^2+t_i$ by $t_i^2$ --- which is the passage from the Appell sequence $B$ to
	the even sequence $\widehat B$ of \S\ref{sec:xiform}, the linear term being exactly what
	$\widehat B$ removes --- the same computation gives the increment
	$\dfrac{m_j-\tfrac12}{p_j}-\dfrac{m_i+\tfrac12}{p_i}$, which is the Sainte-Lagu\"e condition.
\end{proof}

\begin{remark}[Two normalisations, two divisor methods]\label{rem:twodivisors}
	Proposition~\ref{prop:c1divisor} is the sharpest link between the two halves of this paper.  The
	quadratic (Mahalanobis) part of $c_1$ alone yields Webster's method; the linear skewness term
	$-\sum_it_i/(2p_i)$ --- present for the standard Appell sequence $B$ and absent for the even one
	$\widehat B$ --- is exactly what upgrades Webster to Jefferson.  So the two Appell normalisations
	of \S\ref{sec:xiform} are not merely a notational convenience: they correspond to the two
	classical divisor methods, and the shift $t_i\mapsto t_i+\tfrac12$ that relates them is the same
	shift that separates the two rounding rules.  This also explains why the mode is a
	\emph{least-squares} object, connecting Theorem~\ref{thm:mode} to the known least-squares
	characterisations of divisor methods \cite{pukelsheim}.
\end{remark}

\section{Ces\`aro averaging on the sub-simplex}\label{sec:cesaro}

	This section is a consequence of Theorem~\ref{thm:expansion}, not a further ingredient of it:
	nothing earlier depends on it, and a reader interested only in the mode and the local expansion
	may stop here.

	\emph{The object studied here is the formal coefficient family, not a probability.}  We write the
	oscillating part of the coefficients in the fractional-part variables $\theta_i=\{Np_i\}$, using
	the \emph{floor reference}
	\[
		t_i=-\theta_i,\qquad\text{i.e.}\qquad m_i=\lfloor Np_i\rfloor .
	\]
	This point is \emph{not} on the slice: $\sum_it_i=-d^\star$, which violates \eqref{eq:tsum}
	whenever $d^\star\ne0$, so $c_k(-\theta)$ is not the coefficient of any multinomial probability.
	We study it because it is the natural fixed reference in which the $c_k$ are sums of single-bin
	functions of $\theta$, which is what makes the averaging below possible.  By
	$B_{k+1}(1-\theta_i)=(-1)^{k+1}B_{k+1}(\theta_i)$, each $c_k$ then has a well-defined
	$\theta$-dependent component; for instance
	$c_1=\tfrac1{12}-\sum_iB_2(\{Np_i\})/(2p_i)$.

	Whether the results below survive the passage to an on-slice reference such as the mode is
	\emph{open}, and we do not claim it.  At the mode $t_i=-\theta_i+\delta_i(\theta)$, where the
	integer vector $\delta$ allocating the defect $d^\star$ is determined by the rank order of
	Remark~\ref{rem:roundup} and hence depends on \emph{all} $n$ coordinates simultaneously.  The
	single-bin structure is then lost, and with it the support argument on which
	Proposition~\ref{prop:invisible} rests.  Numerically the two references do not agree even in sign.
	For $n=3$ and the generic triple
	$\mathbf p=\bigl(\tfrac{\sqrt2}4,\tfrac{\pi}{12},1-\tfrac{\sqrt2}4-\tfrac{\pi}{12}\bigr)$, the
	Ces\`aro mean of the all-bins term $\prod_iB_2(t_i+1)$ over $N\le3\cdot10^5$ is
	\[
		2.6455\cdot10^{-4}\ \ \text{at the floor reference}\qquad\text{against}\qquad
		-6.010\cdot10^{-3}\ \ \text{at the mode},
	\]
	the former agreeing with the value $1/3780$ of Proposition~\ref{prop:allbins} to six digits.  A
	treatment of the on-slice averages would be a genuinely different computation.

	The coefficients \eqref{eq:cklist}--\eqref{eq:Qjlist} thus oscillate with the fractional-part
	vector
	\[
		\theta(N)=(\{Np_1\},\dots,\{Np_n\}).
	\]
	Because $\sum_ip_i=1$, the sum $\sum_i\{Np_i\}=d^\star(N)$ is always an integer, so $\theta(N)$ is
	confined to the \emph{sub-torus}
	\begin{equation}\label{eq:subtorus}
		H:=\Bigl\{\theta\in\mathbb T^n:\ \textstyle\sum_i\theta_i\equiv0\ (\mathrm{mod}\ 1)\Bigr\},
		\qquad \mathbb T=\R/\Z ,
	\end{equation}
	an $(n-1)$-dimensional set --- the torus analogue of the fractional-part sub-simplex of
	Remark~\ref{rem:combinatorics}.  Note that $H$ arises from the normalisation $\sum_ip_i=1$, not
	from the slice constraint $\sum_im_i=N$; the two are different conditions.  We call $\mathbf p$
	\emph{generic} when
	\begin{equation}\label{eq:generic}
		\{v\in\Z^n:\ v\cdot\mathbf p\in\Z\}=\Z\cdot(1,\dots,1),
	\end{equation}
	i.e.\ the only integer relation satisfied by $\mathbf p$ is $\sum_ip_i=1$ and its multiples.  For
	such $\mathbf p$, Weyl's theorem makes $\theta(N)$ equidistributed on $H$ with Haar measure
	$\mu_H$, so the Ces\`aro mean of any coefficient is its $\mu_H$-average; \eqref{eq:generic} is
	exactly the condition used in the character computation of
	Proposition~\ref{prop:allbins}.

\begin{proposition}[The sub-torus is invisible below order $N^{-n}$]\label{prop:invisible}
	Assume $\mathbf p$ generic in the sense of \eqref{eq:generic}, so that $\theta(N)$ equidistributes on $H$ and the
	Ces\`aro mean of a coefficient is its $\mu_H$-average.  Then, for every $j\le n-1$,
	\[
		\E_{\mu_H}\bigl[Q_j(\theta)\bigr]=\E_{\mathrm{unif}^{\otimes n}}\bigl[Q_j(\theta)\bigr],
	\]
	where $\mathrm{unif}^{\otimes n}$ is the product of $n$ independent uniforms on $[0,1)$.  Thus the
	constraint $\sum_i\theta_i\in\Z$ does not affect the averaged coefficients
	$Q_1,\dots,Q_{n-1}$.
\end{proposition}

\begin{proof}
	$H$ is the kernel of $\theta\mapsto\sum\theta_i$; parametrising by any $n-1$ coordinates leaves
	them free and uniform on $\mathbb T^{n-1}$, so any $\le n-1$ of the $\theta_i$ are jointly uniform.
	A product $\prod_{i\in S}(\cdots)$ with $|S|\le n-1$ therefore factors.  Finally $Q_j$ is a Bell
	polynomial in $c_1,\dots,c_j$, so each monomial is a product $c_{r_1}\cdots c_{r_\ell}$ with
	$r_1+\dots+r_\ell=j$.  Since every $c_r=\sum_ip_i^{-r}B_{r+1}(t_i+1)$ is a sum over single bins,
	expanding such a product gives terms involving at most $\ell\le j$ distinct bins.  For $j<n$ no
	term binds all $n$ coordinates, so the $\mu_H$-average factors through order $N^{-(n-1)}$.
\end{proof}

Equivalently, the first $n-1$ averaged coefficients are obtained by removing the Bernoulli
oscillation bin by bin, as in the one-dimensional averaging of \cite[\S8]{be_paper23}.  The first
coefficient where the sub-simplex constraint can enter is treated next.

\begin{proposition}[The first all-bins correlation]\label{prop:allbins}
	The first term binding all $n$ coordinates is the $\tfrac1{n!}c_1^{\,n}$ contribution to $Q_n$,
	carrying $\prod_iB_2(\{Np_i\})$.  Its $\mu_H$-average is nonzero:
	\begin{equation}\label{eq:allB2}
		\E_{\mu_H}\Bigl[\prod_{i=1}^nB_2(\{Np_i\})\Bigr]
		=\frac{(-1)^{n+1}2^nB_{2n}}{(2n)!}
		=\frac{2\zeta(2n)}{(2\pi^2)^n},
	\end{equation}
	whereas the product average under independent coordinates gives $0$.  More generally, with
	$\widehat B_k(m):=-k!/(2\pi im)^k$ the Fourier coefficients of $B_k(\{\cdot\})$ and $K:=\sum_ik_i$,
	\begin{equation}\label{eq:mixed}
		\begin{aligned}
		\E_{\mu_H}\Bigl[\prod_{i=1}^nB_{k_i}(\{Np_i\})\Bigr]
		&=\sum_{m\neq0}\prod_{i=1}^n\widehat B_{k_i}(m)\\
		&=\begin{cases}
		\dfrac{2(-1)^{\,n-K/2}\bigl(\prod_ik_i!\bigr)\zeta(K)}{(2\pi)^{K}},
			& K\ \text{even},\\[4pt]
		0, & K\ \text{odd}.
		\end{cases}
		\end{aligned}
	\end{equation}
\end{proposition}

	For $n=2,3,4$ the values in \eqref{eq:allB2} are
	$\tfrac1{180}$, $\tfrac1{3780}$ and $\tfrac1{75600}$.

\begin{proof}
	Only the ``diagonal'' characters $m(1,\dots,1)$, $m\in\Z$, are trivial on $H=\ker(\sum)$, so
	\[
		\E_{\mu_H}[e^{2\pi i\sum m_i\theta_i}]=\mathbf 1[m_1=\dots=m_n].
	\]
	Expanding each Bernoulli in its Fourier series and using $\widehat B_k(0)=0$ leaves
	$\sum_{m\neq0}\prod_i\widehat B_{k_i}(m)
	=\frac{\prod_i(-k_i!)}{(2\pi i)^K}\sum_{m\neq0}m^{-K}$, which is \eqref{eq:mixed}
	($\sum_{m\neq0}m^{-K}=2\zeta(K)$ for even $K$, $0$ for odd $K$; $(2\pi i)^K=(2\pi)^K(-1)^{K/2}$).
	For $k_i\equiv2$: $K=2n$, giving \eqref{eq:allB2}.
\end{proof}

\begin{corollary}[The order at which the sub-torus first bites is exactly $n$]\label{cor:ordern}
	Under the hypotheses above, $\E_{\mu_H}[Q_j]=\E_{\mathrm{unif}^{\otimes n}}[Q_j]$ for $j\le n-1$,
	while
	\[
		\E_{\mu_H}[Q_n]-\E_{\mathrm{unif}^{\otimes n}}[Q_n]
		=\frac{(-1)^n}{2^n\prod_ip_i}\cdot\frac{2\zeta(2n)}{(2\pi^2)^n}\ \neq\ 0 .
	\]
\end{corollary}

\begin{proof}
	The first statement is Proposition~\ref{prop:invisible}.  For the second, expand $Q_n$ into
	monomials $c_{r_1}\cdots c_{r_\ell}$ with $\sum r_i=n$.  Every monomial with $\ell<n$ factors, and
	every monomial with $\ell=n$ other than $\tfrac1{n!}c_1^{\,n}$, involves at most $n-1$ distinct
	bins after expansion, so its two averages agree by the argument of
	Proposition~\ref{prop:invisible}.  The only surviving difference comes from the all-bins part of
	$\tfrac1{n!}c_1^{\,n}$, which carries $\prod_i\bigl(-B_2(\theta_i)/(2p_i)\bigr)$; its
	$\mu_H$-average is given by \eqref{eq:allB2} while its independent-coordinate average vanishes,
	since $\int_0^1B_2=0$.
\end{proof}

\begin{remark}[Rational $\mathbf p$: the orbit does not fill the sub-torus]\label{rem:rational-cesaro}
	For rational $\mathbf p=(a_1/b,\dots,a_n/b)$ with $\sum a_i=b$, $\theta(N)$ runs over the
	\emph{cyclic} orbit $\{N\mathbf a/b\bmod1\}$, which is one-dimensional and, for $n\ge3$, a proper
	subset of the $(n-1)$-dimensional sub-torus $H$.  The period-average is a resonance sum
	\begin{align*}
		\frac1P\sum_{N}\prod_iB_2(\{Na_i/b\})
		&=\sum_{\substack{m\in\Z^n\\ m\cdot\mathbf a\equiv0\ (b)}}\ \prod_i\widehat B_2(m_i)\\
		&=\underbrace{\frac{(-1)^{n+1}2^nB_{2n}}{(2n)!}}_
			{\substack{\text{diagonal }m\propto(1,\dots,1)\\ \text{always resonant}}}
		+\ (\text{additional resonances}),
	\end{align*}
	the diagonal being always present because $\sum a_i=b\equiv0$.  The additional part is
	arithmetic-dependent: it is positive in the tested cases, tends to $0$ along good rational
	approximations $b\to\infty$ (below), but is \emph{not} monotone in $b$
	(e.g.\ $(31,37,43)/111$ has a larger correction than $(13,17,29)/59$).  This is the genuine departure
	from the binomial: for $n=2$ both $H$ and the orbit are one-dimensional, so the orbit is a finite
	cyclic subgroup of $H$ (equal to it only in the limit $b\to\infty$), and rational and irrational
	both reduce to Raabe's $B_j(h)\mapsto b^{-j}B_j$ of \cite[\S8]{be_paper23}; for $n\ge3$ the orbit
	is one-dimensional inside an $(n-1)$-dimensional $H$, hence a thin subset, and no such clean
	substitution holds.  (Even at $n=2$ the resonance set is
	$\{m:\ m_1\equiv m_2\ \mathrm{mod}\ b/\!\gcd(a_1,b)\}$, not merely the diagonal.)
\end{remark}

\paragraph{Rational-period examples.}  For $n=3$, the sub-torus value in
	\eqref{eq:allB2} is $\tfrac1{3780}$.  The following period averages show the additional
	resonance term in Remark~\ref{rem:rational-cesaro} for several rational vectors
	$\mathbf p=(a_1/b,a_2/b,a_3/b)$:
	\[
	\begin{array}{c|ccc}
		(a_1,a_2,a_3)/b & (2,3,5)/10 & (5,7,11)/23 & (13,17,29)/59\\\hline
		\text{additional resonance}
		& 5.9\!\cdot\!10^{-4} & 1.0\!\cdot\!10^{-4} & 3.1\!\cdot\!10^{-5}
	\end{array}
	\]
	For comparison, along rational approximations to the fixed generic triple
	$\mathbf p=\bigl(\tfrac{\sqrt2}4,\ \tfrac{\pi}{12},\ 1-\tfrac{\sqrt2}4-\tfrac{\pi}{12}\bigr)$,
	the non-diagonal resonance contribution decays as follows:
	\[
	\begin{array}{c|ccc}
		b & 10^3 & 10^4 & 10^5\\\hline
		\text{additional resonance}
		& 1.2\!\cdot\!10^{-6} & 2.5\!\cdot\!10^{-8} & 3.7\!\cdot\!10^{-11}
	\end{array}
	\]

\section{Concluding remarks}\label{sec:conclusion}

	The two components of the paper use elementary methods: separable lattice concavity for the mode
	and a gamma-quotient Stirling development for the expansion.  The new points are the interpretation
	of the mode as a Jefferson apportionment, with its sharp quota violation, and the closed all-orders
	Bernoulli expansion for the multinomial local mass, with the binomial and central-multinomial cases
	as specialisations.

	Several questions are left open.  In the even (entropy) normalisation of \S\ref{sec:xiform} the
	multiplicative coefficients $\widehat Q_j$ have not been written out, and it is not clear whether
	the size-biased third normalisation of \cite{be_paper23} has an $n$-bin analogue.  For the
	Ces\`aro averages, the full $\mu_H$-average of $Q_j$ for $j\ge n$ --- combining the all-bins
	correlation \eqref{eq:allB2} with the lower-order product terms --- and a closed description of the
	rational period-average through its resonance lattice (Remark~\ref{rem:rational-cesaro}) remain to
	be done.

\end{document}